\documentclass[11pt,reqno]{amsart}

\usepackage{amsmath,amsthm,amssymb,color}
\usepackage[tmargin=1.2in,bmargin=1.2in,rmargin=1.2in,lmargin=1.2in]{geometry}
\usepackage[breaklinks=true]{hyperref}

\theoremstyle{plain}

\newtheorem{theorem}{Theorem}[section]

\newtheorem{corollary}[theorem]{Corollary}
\newtheorem{proposition}[theorem]{Proposition}
\newtheorem{lemma}[theorem]{Lemma}

\theoremstyle{definition}

\newtheorem{definition}[theorem]{Definition}
\newtheorem{example}[theorem]{Example}
\newtheorem{remark}[theorem]{Remark}

\theoremstyle{remark}

\numberwithin{equation}{section}
\numberwithin{table}{section}

\DeclareMathOperator{\Aut}{Aut}

\begin{document}
\title{Maps preserving the sum-to-difference ratio in characteristic $p$}
\date{\today}

\author{Sunil Chebolu}
\address[S.~Chebolu]{Department of Mathematics, Illinois State University,
Normal, IL 61761, USA}
\email{\tt schebol@ilstu.edu}

\author{Apoorva Khare}
\address[A.~Khare]{Department of Mathematics, Indian Institute of
Science, Bangalore 560012, India}
\email{\tt khare@iisc.ac.in}

\author{Anindya Sen}
\address[A.~Sen]{Accountancy \& Finance Department, 
University of Otago, Dunedin 9016, NZ}
\email{\tt anindya.sen@otago.ac.nz}



\begin{abstract}
Given a field $\mathbb{F}$, we introduce a novel group $SD(\mathbb{F})$
of its self-maps: the solutions $f \colon \mathbb{F} \twoheadrightarrow
\mathbb{F}$ to the functional equation $f \left( (x+y)/(x-y) \right) =
(f(x) + f(y))/(f(x) - f(y))$ for all $ x \neq y$ in $\mathbb{F}$. We
compute this group for all fields algebraic over $\mathbb{F}_p$. In
particular, this group distinguishes $\mathbb{F}_5$ among all finite
fields $\mathbb{F}_q$, and in fact among all subfields of
$\overline{\mathbb{F}_q}$.
\end{abstract}

\thanks{A.K.\ was partially supported by a Shanti Swarup Bhatnagar Award
from CSIR (Govt.\ of India).}

\maketitle

\section{Introduction}

One of the most fundamental and famous functional equations is named
after Cauchy:
\begin{equation}\label{Ecauchy}
f : \mathbb{R} \to \mathbb{R}, \qquad f(x+y) = f(x) + f(y).
\end{equation}

This has been much studied in the literature, e.g.\ ``back-to-back'' by
Sierpi\'nsky~\cite{Sierpinsky} and Banach~\cite{Banach} in 1920; by now
there are several monographs as well, e.g.\
\cite{Aczel-Dhombres,Czerwik}. It is easy to check that~\eqref{Ecauchy}
implies $f$ is $\mathbb{Q}$-linear -- and hence linear if $f$ is
continuous. However, without continuity one arrives at non-measurable,
wild solutions -- for instance, the graph of every discontinuous map $f$
satisfying~\eqref{Ecauchy} is dense in $\mathbb{R}^2$.

A slight modification of the Cauchy functional equation is
\[
f(x-y) = f(x)-f(y), \qquad \forall x,y \in \mathbb{R},
\]
and it is a short exercise to check that the solutions of this coincide
with those of~\eqref{Ecauchy}.

This led us to a natural question: what if we combine both properties in
one functional equation? Namely, classify all maps $f : \mathbb{R} \to
\mathbb{R}$ such that
\begin{equation}\label{Esumdiff}
f \left( \frac{x+y}{x-y} \right) = \frac{f(x) + f(y)}{f(x) - f(y)},
\qquad \forall x \neq y.
\end{equation}

A first observation is that $f(y) \equiv y$ is a solution
of~\eqref{Esumdiff}, but $f(y) = cy$ is not for any $c \neq 1$, so the
problem is different from that of solving the Cauchy functional
equation~\eqref{Ecauchy}. Nevertheless, one can show that $f(y) \equiv y$
is the only solution to~\eqref{Esumdiff}.

We next observe that
(a)~equation~\eqref{Esumdiff} makes sense over any field $\mathbb{F}$,
not just $\mathbb{R}$; and
(b)~every field automorphism $f : \mathbb{F} \to \mathbb{F}$
satisfies~\eqref{Esumdiff}.
Thus, in recent work~\cite{CKS} we initiated the study
of~\eqref{Esumdiff} -- over fields of characteristic zero. We showed that
for $\mathbb{F} = \mathbb{Q}$ or $\mathbb{R}$, the only solution
to~\eqref{Esumdiff} is $f(y) \equiv y$. Moreover, for arbitrary subfields
$\mathbb{F}' \subseteq \mathbb{R}$, the only continuous self-maps of
$\mathbb{F} = \mathbb{F}'$ or $\mathbb{F}'(i)$ that
satisfy~\eqref{Esumdiff} are $f(z) \equiv z$ or $\overline{z}$.
These and other results shown in~\cite{CKS} are summarized in
Table~\ref{Ttable1}. In all of these cases (with $\mathbb{F} =
\mathbb{Q}_p$ in the last row), it is remarkable that~\eqref{Esumdiff}
\textit{guarantees the surjectivity of $f$ on $\mathbb{F}$}.

\begin{table}[ht]
\begin{tabular}{| c | c | c | c | c |}
\hline
$\#$ & \textit{Domain} & \textit{Codomain} & \textit{Hypotheses} &
\textit{Solution set}\\
\hline\hline

1 & $\mathbb{Q}$ & $\mathbb{C}$ & & ${\rm id}$\\
\hline
2 & $\mathbb{F} \subseteq \mathbb{R}$ & $\mathbb{R}$ &
$\mathbb{F} \cap (0,\infty)$ closed under $\sqrt{\cdot}$ &
{\rm id}\\
\hline
3 & $\mathbb{F}' \subseteq \mathbb{R}$ & $\mathbb{R}$ &
$f$ continuous & {\rm id}\\
\hline\hline

4 & $\mathbb{Q}(i)$ & $\mathbb{C}$ & & $z$ or $\overline{z}$\\
\hline
5 & $\mathbb{F}(i), \ \mathbb{F} \subseteq \mathbb{R}$ & $\mathbb{C}$ &
$\mathbb{F} \cap (0,\infty)$ closed under $\sqrt{\cdot}$, &
$z$ or $\overline{z}$\\
& & & $f(\mathbb{F}) \subseteq \mathbb{R}$ & \\
\hline
6 & $\mathbb{F}'(i), \ \mathbb{F}' \subseteq \mathbb{R}$ & $\mathbb{C}$ &
$f$ continuous &
$z$ or $\overline{z}$\\
\hline\hline

7 & $\mathbb{F} \subseteq \mathbb{Q}_p$, $p \geq 2$ & $\mathbb{Q}_p$ &
$f$ continuous & ${\rm id}$\\
\hline

\end{tabular}\vspace*{1mm}
\caption{Summary of results in~\cite{CKS}, on solutions $f$
to~\eqref{Esumdiff} in characteristic zero}\label{Ttable1}
\end{table}

One upshot of this table is rows~$5,6$, which yield that
(a)~\textit{the $\mathbb{R}$-preserving automorphisms of $\mathbb{C}$ are
precisely the same as the $\mathbb{R}$-preserving solutions
to~\eqref{Esumdiff}}; and
(b)~\textit{the continuous automorphisms of $\mathbb{C}$ are precisely
the same as the continuous solutions to~\eqref{Esumdiff}.}
This suggests that the solutions to~\eqref{Esumdiff} parallel the
automorphism group (under composition). In turn, this motivates us to
introduce the following notion.

\begin{definition}
Given a field $\mathbb{F}$, define the \textit{SD-group}\footnote{``SD''
stands for sum-difference, in the absence of a better name.} to be
\[
SD(\mathbb{F}) := \{ f : \mathbb{F} \to \mathbb{F} \ | \ f \textit{ is
onto, and satisfies } \eqref{Esumdiff} \},
\]
and the maps satisfying~\eqref{Esumdiff} (but not necessarily surjective)
to be \textit{SD-maps}.
\end{definition}

\noindent To the best of our knowledge, $SD(\mathbb{F})$ is a novel
construction.

We thus specialize the table (by equating the codomain to the domain)
into a summary.

\begin{theorem}
For any field $\mathbb{F}$, $SD(\mathbb{F})$ is a group (under
composition) containing $\Aut(\mathbb{F})$ as a subgroup. Moreover:
\begin{enumerate}
\item $SD(\mathbb{F}) = {\rm Aut}(\mathbb{F})$ is trivial, for
$\mathbb{F} = \mathbb{Q}$ or if $\mathbb{F}$ is a real subfield with
$\mathbb{F} \cap (0,\infty)$ closed under square roots.\footnote{This
includes the reals, the real algebraic numbers $\mathbb{R} \cap
\overline{\mathbb{Q}}$, and the real constructible numbers -- via ruler
and compass, or via taking iterated square roots. We also quickly show
that if $\mathbb{F} \cap (0,1)$ is closed under $\sqrt{\cdot}$ then
$\varphi \equiv {\rm id}_{\mathbb{F}}$ is the only automorphism.
Indeed, if $x \in \mathbb{F}_+$ then so is $\sqrt{x}$, whence $\varphi(x)
= \varphi(\sqrt{x})^2 \geq 0$. Thus $\varphi$ is order-preserving; now
since it fixes $\mathbb{Q}$ which is dense in $\mathbb{F}$ in the order
topology of $\mathbb{R}$, it also fixes $\mathbb{F}$.}
For such fields, we also have $SD(\mathbb{F}(i)) = (1) =
SD(\mathbb{Q}(i))$.

\item Given a topological field $\mathbb{F}$, define
$SD_{cont}(\mathbb{F})$ to be the subgroup of $SD(\mathbb{F})$ consisting
of continuous maps, and similarly, ${\rm Aut}_{cont}(\mathbb{F})$. If
$\mathbb{F}$ is any subfield of $\mathbb{R}$, then:
(a)~$SD_{cont}(\mathbb{F}) = {\rm Aut}_{cont}(\mathbb{F}) = (1)$, and
(b)~$SD_{cont}(\mathbb{F}(i)) = {\rm Aut}_{cont}(\mathbb{F}(i)) =
\mathbb{Z} / 2 \mathbb{Z}$.
\end{enumerate}
\end{theorem}

From the above statements one could have wondered if the SD-group is not
a new notion at all, but equals ${\rm Aut}(\mathbb{F})$. However, this is
false, and reveals $SD(\mathbb{F})$ is a novel group:

\begin{example}\label{Ex5}
Consider $f(w) = w^3$ over $\mathbb{F}_5 = \mathbb{Z} / 5 \mathbb{Z}$;
this map is obviously multiplicative, and fixes $0,1,4 \mod 5$; but it
interchanges $2,3 \mod 5$. So it is not additive (since $f(1+1) \neq
f(1)+f(1)$), hence not an automorphism. We claim $f$
satisfies~\eqref{Esumdiff}. Indeed, setting $y=0$ trivially works. If $y
\neq 0$, by multiplicativity one can replace $(x,y)$ by $(w := xy^{-1},
1)$; now we need to verify that
\[
\frac{(w+1)^3}{(w-1)^3} = \frac{w^3+1}{w^3-1}, \qquad \forall w \neq 1
\mod 5.
\]
This is done case by case, and shows that ${\rm id} \neq f \in
SD(\mathbb{F}_5)$. \qed
\end{example}

Thus, a closer look at~\eqref{Esumdiff} over prime fields -- more
generally, finite fields -- is in order, and this is the goal of the
present work.

At the outset, we point out that characteristic~$2$ is pathological.
Namely, the prime field $\mathbb{F}_2$ is the \textit{only} one that
satisfies~\eqref{Esumdiff}. For all other fields, the result fails as
badly as possible:

\begin{proposition}\label{Pchar2}
Suppose $\mathbb{F}, \widetilde{\mathbb{F}}$ are fields of characteristic
$2$. Then $f : \mathbb{F} \to \widetilde{\mathbb{F}}$
satisfies~\eqref{Esumdiff} if and only if $f(1) = 1$ and $f$ is
injective.
\end{proposition}

Thus, if $\mathbb{F} = \mathbb{F}_2$ then $f$ is the identity; but for
every other field, any injection $: \mathbb{F} \to \mathbb{F}$ that fixes
$1$ is a solution of~\eqref{Esumdiff}. For example, any set-bijection of
$\mathbb{F}_{2^k}$ that fixes $1$ is an SD-map.

\begin{proof}
The equation~\eqref{Esumdiff} reduces to precisely: (a)~$f$ is injective
and (b)~$f(1) = 1$, since $x+y=x-y\neq0$ and $f(x) + f(y) = f(x) - f(y)
\neq 0$ for $x \neq y \in \mathbb{F}$.
\end{proof}

Thus, our goal henceforth is to consider the question ``\textit{Is
$SD(\mathbb{F}) = \Aut(\mathbb{F})$?}'' in odd characteristic. Given the
counterexample of $\mathbb{F}_5$, a first question is: for how many (and
which) other finite fields of odd characteristic is the answer negative.
Remarkably, we show that $\mathbb{F}_5$ is the only exception among prime
finite fields -- and more strongly, among all odd finite fields. Even
stronger is our main result: $\mathbb{F}_5$ is the \textit{only exception
among all algebraic extensions} of $\mathbb{F}_p$:

\begin{theorem}\label{Tmain}
Let $\mathbb{F}$ be a field of odd characteristic $p>2$ that is algebraic
over $\mathbb{F}_p$. Then
\[
SD(\mathbb{F}) = \begin{cases}
\Aut(\mathbb{F}) & \text{ if } \mathbb{F} \ne \mathbb{F}_5, \\
\{ {\rm id}, w^3 \} \cong \mathbb{Z} / 2 \mathbb{Z} \qquad
& \text{ if } \mathbb{F} = \mathbb{F}_5  .
\end{cases}
\]
\end{theorem}

\noindent (In particular, this includes the fields
$\overline{\mathbb{F}_p}$ for all $p>2$, as well as $\mathbb{F}_{5^\ell}$
for $\ell \geq 2$.)
Thus, Theorem~\ref{Tmain} distinguishes $\mathbb{F}_5$ among all finite
fields, and more. In fact, the next result, which is in terms of the
SD-map $w \mapsto w^3$, distinguishes $\mathbb{F}_5$ among \textit{all}
fields of characteristic $\neq 2$.

\begin{theorem}\label{Tiff5}
Let $\mathbb{F}$ be a field of characteristic not $2$. The map $w \mapsto
w^3$ is in $SD(\mathbb{F}) \setminus \Aut(\mathbb{F})$, if and only if
$\mathbb{F} = \mathbb{F}_5$.
\end{theorem}

The majority of the paper is devoted to showing
Theorems~\ref{Tmain} and~\ref{Tiff5} (and some variants). In the
Appendix, we also completely classify the SD-maps of the $p$-adic
numbers, for all $p$. Given that these and the reals form all completions
of $\mathbb{Q}$, we have:

\begin{theorem}
Let $\mathbb{F}$ be any field that is a completion of $\mathbb{Q}$. Then 
$SD(\mathbb{F}) = {\rm Aut}(\mathbb{F}) = \{ {\rm id} \}$.
\end{theorem}

\section{A preliminary result, and some additional ones}

The proof of Theorem~\ref{Tmain} has three parts; the first can be
formulated and shown over arbitrary fields (of characteristic not $2$):

\begin{theorem}\label{Tmult}
Suppose $\mathbb{F}, \widetilde{\mathbb{F}}$ are arbitrary fields of
characteristic not $2$, and $f : \mathbb{F} \to \widetilde{\mathbb{F}}$
is an SD-map, i.e.\ $f$ satisfies~\eqref{Esumdiff}. Then $f$ is
injective, multiplicative, and odd, and fixes $0,1$.
\end{theorem}

Thus, SD-maps are ``close'' to being field homomorphisms: they just need
not be additive.

\begin{proof}
This was shown in~\cite{CKS} for fields of characteristic zero; for
self-containedness, we reproduce the short proof here.
Clearly, $f$ is one-to-one on $\mathbb{F}$, since if $x \neq y$ then the
denominator on the right in~\eqref{Esumdiff} must be nonzero.
Next, we claim that $f$ fixes $0$ and $1$. Indeed, for all $x \in
\mathbb{F}^\times$ we have:
\begin{equation}
f(1) = f \left( \frac{x+0}{x-0} \right) = \frac{f(x) + f(0)}{f(x) - f(0)}
= 1 + 2 \frac{f(0)}{f(x) - f(0)}.
\end{equation}

This implies: $\displaystyle \frac{f(1)-1}{2} = \frac{f(0)}{f(x)-f(0)}$
(for the left-hand side to make sense, $2 \neq 0$ in
$\widetilde{\mathbb{F}}$). That is, the right-hand side is a constant
function on $\mathbb{F} \setminus \{ 0 \}$ (which has at least $2$
elements), since the left side is. Since $f$ is one-to-one, we get $f(0)
= 0$ and this in turn yields $f(1) = 1$; moreover, $f$ is nonzero on
$\mathbb{F} \setminus \{ 0 \}$.

From $f(0) = 0$, it is immediate that $f$ is odd:
\begin{equation}
0 = f(0) = f \left( \frac{y+(-y)}{y-(-y)} \right) = \frac{f(y) +
f(-y)}{f(y)-f(-y)}, \qquad \forall y > 0.
\end{equation}

Finally, we claim that $f$ is multiplicative on $\mathbb{F}$. For this,
let $1 \neq k, x \neq 0$ be in $\mathbb{F}$. Then $kx \neq x$, so
\[
f \left( \frac{k+1}{k-1} \right) = \frac{f(kx) + f(x)}{f(kx) - f(x)} = 1
+ 2 \frac{f(x)}{f(kx)-f(x)}.
\]
Since $f(\frac{k+1}{k-1}) \neq 1$ and $f(x) \neq 0$ (because $f$ is
one-to-one), it follows that
\begin{equation}
\frac{2}{f(\frac{k+1}{k-1}) - 1} = \frac{f(kx)-f(x)}{f(x)} =
\frac{f(kx)}{f(x)} - 1 \quad \implies \quad \frac{f(kx)}{f(x)} = 1 +
\frac{2}{f(\frac{k+1}{k-1}) - 1}, \quad \forall x \neq 0.
\end{equation}

The right-hand side is independent of $x$, so it equals the left-hand
side at $x=1$, which is precisely $f(k)$. Thus, we obtain over
$\mathbb{F}$:
\begin{equation}\label{E1}
f(kx) = f(k) f(x), \qquad \forall k \neq 1, x \neq 0.
\end{equation}
If instead $k=1$ or $x=0$ then~\eqref{E1} is direct, since $f$ fixes
$0,1$. Thus, $f$ is multiplicative.
\end{proof}

With Theorem~\ref{Tmult} at hand, we proceed. The reader who wishes to
see just the proof of Theorem~\ref{Tmain}, should move directly to the
next section -- where we provide a direct path to the result. In the rest
of this section, we deduce additional information about SD-maps in
arbitrary characteristic, beginning with asking if the existence of an
SD-map necessitates the characteristics being equal. This is false for
small $p$, as we now explain beginning with $p=3$:

\begin{example}
Suppose $\mathbb{F} = \mathbb{F}_3$ and $\widetilde{\mathbb{F}}$ is any
field of characteristic not $2$. One can use~\eqref{Esumdiff} thrice, to
explicitly verify that the map $f(w) = w$ for $w = -1, 0, 1$ is an SD-map
-- which is not a field homomorphism unless ${\rm
char}(\widetilde{\mathbb{F}}) = 3$. \qed
\end{example}

The next case is $p=5$:

\begin{proposition}
There exists an SD-map $f : \mathbb{F}_5 \to \widetilde{\mathbb{F}}$
(i.e., a solution to~\eqref{Esumdiff}) if and only if the latter field
admits a primitive fourth root of unity.
\end{proposition}

Thus, all finite fields of size $1 \mod 4$, and all characteristic zero
fields containing the Gaussian rationals $\mathbb{Q}(i)$, have SD-maps
into them from $\mathbb{F}_5$.

\begin{proof}
By Theorem~\ref{Tmult}, $f$ is multiplicative and injective. Since $2
\mod 5$ is a primitive fourth root of unity, the ``only if'' assertion
follows. Conversely, let $\zeta \in \widetilde{\mathbb{F}}$ be such a
primitive root, and define $f : \mathbb{F}_5 \to \widetilde{\mathbb{F}}$
via:
\[
f(0) = 0, \quad f(\pm 1) = \pm 1, \quad f(\pm 2) = \pm \zeta.
\]
Note that $\zeta \neq \pm \zeta$ since then $1 = -1$ by cancelling, in
which case $w^4 - 1 = (w^2-1)(w^2+1) = (w^2-1)^2 = (w-1)^4$, and so
$\widetilde{\mathbb{F}}$ does not have nontrivial (primitive) fourth
roots of unity.

We now check that $f$ is an SD-map. Indeed, if $x$ or $y$ is zero
then~\eqref{Esumdiff} follows because $f$ is odd and fixes $1$. Else
since $f$ is multiplicative, we can replace $(x,y)$ by $(w = xy^{-1},1)$
as in Example~\ref{Ex5}, and it suffices to show:
\[
\frac{f(w+1)}{f(w-1)} = \frac{f(w)+1}{f(w)-1}, \qquad \forall w \neq 0, 1.
\]
This is trivial for $w=-1$, and easily checked for $w=\pm 2$ since
$\zeta^2 = -1$ (and $-2 \neq 0$):
\begin{align*}
\frac{f(2)+1}{f(2)-1} = &\ \frac{\zeta+1}{\zeta-1} =
\frac{(\zeta+1)^2}{\zeta^2-1} = \frac{2 \zeta}{-2} =
-\zeta = f(-2) = \frac{f(2+1)}{f(2-1)},\\
\frac{f(-2)+1}{f(-2)-1} = &\ \frac{-\zeta+1}{-\zeta-1} =
\frac{(-\zeta+1)^2}{\zeta^2-1} = \frac{-2 \zeta}{-2} =
\zeta = f(2) = \frac{f(-2+1)}{f(-2-1)}. \qedhere
\end{align*}
\end{proof}

Thus there exist SD-maps from $\mathbb{F}_p$ to e.g.\ $\mathbb{Q}(i)$,
for $p=3,5$. However, with these two exceptions (and $p=2$), the result
is true in all other characteristics. More strongly, every SD-map fixes
the prime subfield:

\begin{theorem}\label{Tchar}
Suppose $f : \mathbb{F} \to \widetilde{\mathbb{F}}$ is an SD-map, with
${\rm char}(\mathbb{F}) \neq 2,3,5$ and ${\rm
char}(\widetilde{\mathbb{F}}) \neq 2$. Then ${\rm char}(\mathbb{F}) =
{\rm char}(\widetilde{\mathbb{F}})$. More strongly, $f$ fixes the common
prime subfield.
\end{theorem}

Thus, the existence of an SD-map $f : \mathbb{F} \to
\widetilde{\mathbb{F}}$ automatically places restrictions on the
characteristic. Notice that this is immediate if $f$ is an automorphism,
which \textit{a posteriori} holds for all finite prime fields
$\mathbb{F}_p$, $p \neq 5$. However, the proof for a (more general)
SD-map is more involved, since we cannot assume $f$ is additive -- by
Example~\ref{Ex5}.

This result also (post-facto) justifies our working in~\cite{CKS} with
SD-maps between fields of characteristic zero: such maps cannot exist if
${\rm char}(\mathbb{F}) = 0 \neq {\rm char}(\widetilde{\mathbb{F}})$.

\begin{proof}
Set $p := {\rm char}(\mathbb{F}) \geq 0$. Since $p=0$ or $p \geq 7$, so
$0 = f(0), 1 = f(1)$, and $f(2), \dots, f(6)$ are pairwise distinct by
injectivity. As a remarkable ``preview'' of what follows: we only need
these values (in fact, not even $f(4)$) to prove the result!

We now explicitly compute $f$ at these (and other) positive integers in
terms of $f(2)$ and via a recurrence relation for $f(n)$. The key
observation is to set $y = 1$ in~\eqref{Esumdiff}:
\[
f \left( \frac{n+1}{n-1} \right) = \frac{f(n)+1}{f(n)-1}, \qquad \forall
n = 2, 3, \ldots \in \mathbb{F}.
\]

Since $f$ is injective and $n \neq 1$, we have $f(n) \neq 1$. By
multiplicativity, we therefore obtain a recurrence relation for the
$f(n)$:
\begin{equation}\label{Erecur}
f(n+1) = f(n-1) \cdot \frac{f(n)+1}{f(n)-1}, \qquad \forall 2 \leq n \in
\mathbb{F},
\end{equation}
where we ``cycle round to $n=p=0$'' if $p>0$.
Using this relation, one can explicitly compute $f(n)$ -- in terms of
$u := f(2)$ -- e.g.,
\begin{align}\label{E5}
\begin{aligned}
f(0) = &\ 0, \quad f(1) = 1, \quad f(2) = u, \quad f(3) =
\frac{u+1}{u-1}, \quad f(4) = u^2, \quad f(5) = \frac{u^2+1}{(u-1)^2},\\
f(6) = &\ u(u^2-u+1).
\end{aligned}
\end{align}
(See Lemma~\ref{Lformula} for a closed-form expression for $f(n)$ for
each $n$.) Now since $f(2) f(3) = f(6)$ and $u \neq 1$, we have
\[
u(u+1) = (u-1)u (u^2-u+1) \quad \implies \quad u(u-2)(u^2+1) = 0.
\]
If $u = f(2)$ does not equal $2$, then (since $u \neq 0$,) $u^2+1 = 0$.
Hence $f(5) = 0$, which contradicts the injectivity of $f$.

We conclude that $f(2) = u = 2$, and hence inductively via~\eqref{Erecur}
that $f(n) = n$ for all $2 \leq n \in \mathbb{F}$.
This observation implies both the assertions in the theorem:
\begin{enumerate}
\item The characteristics agree. Indeed, if $p=0$ then $f(1) = 1, f(1+1)
= 1+1, \dots$ are pairwise distinct in $\widetilde{\mathbb{F}}$, whence
${\rm char}(\widetilde{\mathbb{F}}) = 0$. Else $p \geq 7$, and $f(0) = 0,
\dots, f(p-1) = p-1$ are pairwise distinct in $\widetilde{\mathbb{F}}$;
and $p = f(p) = f(0) = 0$. Hence ${\rm char}(\widetilde{\mathbb{F}}) =
p$.

\item The prime field is fixed. Indeed, the above observation (and the
previous point) show this if $p \geq 7$; and if $p=0$ then the oddness of
$f$ implies that it fixes $\mathbb{Z} \subset \mathbb{F}$; now we are
done by multiplicativity. \qedhere
\end{enumerate}
\end{proof}

\begin{remark}
If $p \neq 7$ then an alternate approach to the one above yields that
$f(n) = n$ for all $n \geq 0$. This was carried out in~\cite{CKS}, and
used that $0 = f(8) - f(2) f(4) = u^2(u-1)(u-2)$. The approach here has
two advantages. First, it only requires ${\rm
char}(\widetilde{\mathbb{F}}) \neq 2$, compared to the assumption ${\rm
char}(\widetilde{\mathbb{F}}) = 0$ in~\cite{CKS}.
Second, unlike~\cite{CKS} the approach here also works for $p=7$. (And for
$p=3,5$ we saw above that the result is false.)
\end{remark}

Remarkably, the proof of Theorem~\ref{Tchar} required only the values
$f(n)$ for $n = 0,1,2,3,5,6$. However, one can write a closed-form
expressions for all $f(n)$ (which looks different in~\eqref{E5} for odd
and even values of $n$). For completeness, we present these formulas to
close this section.

\begin{lemma}\label{Lformula}
Define the sequence of one-variable polynomials $p_0, p_1, \ldots$ in
$\mathbb{Z}[u]$ via:
\[
p_0(u) := 1, \qquad p_{k+1}(u) := 2 + (u-1) p_k(u).
\]
Then for all $k \geq 0$, we have:
\begin{align}\label{Egeom}
p_k(u) = &\ 2 + 2(u-1) + 2(u-1)^2 + \cdots + 2(u-1)^{k-1} \ + (u-1)^k,\\
f(2k+1) = &\ \frac{p_k(u)}{(u-1)^k},\notag\\
f(2k+2) = &\ p_{k+1}(u)-1 = 1 + (u-1)p_k(u).\notag
\end{align}
\end{lemma}

\begin{proof}
By induction on $k$. The $k=0$ case of~\eqref{Egeom} is trivial.
For the induction step, we omit the easy calculation for $p_{k+1}$ in
terms of $p_k$, given that $p_k(u) + (u-1)^k$ is a geometric series.

For the $f$-values, if we know $f(2k-1), f(2k)$ for
some $k \geq 1$, then
\[
f(2k+1) = f(2k-1) \cdot \frac{f(2k)+1}{f(2k)-1} =
\frac{p_{k-1}(u)}{(u-1)^{k-1}} \cdot
\frac{1+(u-1)p_{k-1}(u)+1}{(u-1)p_{k-1}(u)} = \frac{p_k(u)}{(u-1)^k}
\]
as claimed. Next,
\[
f(2k+2) = f(2k) \cdot \frac{f(2k+1)+1}{f(2k+1)-1} = (p_k(u)-1) \cdot
\frac{p_k(u) + (u-1)^k}{p_k(u) - (u-1)^k},
\]
and we are to show this equals $1 + (u-1)p_k(u)$, i.e.,
\[
(p_k(u) - 1)(p_k(u) + (u-1)^k) = (p_k(u)-(u-1)^k)(1 + (u-1)p_k(u)).
\]
We verify this on the level of polynomials: adding $(u-1)^k$ to both
sides, it suffices to show the equality of the remaining terms divided
by $p_k(u)$. That is, we need to verify that
\[
((u-1) - 1)p_k(u) = (u-1)^k - 1 + (u-1)^{k+1} - 1,
\]
and this follows from~\eqref{Egeom} as $p_k(u)$ is a sum of two geometric
series in $u-1$.
\end{proof}

\section{Proof of the main result}

We now return to the proof of Theorem~\ref{Tmain}. Interestingly, the
unique exceptional field $\mathbb{F}_5$ emerges naturally from our proof.
Also, as promised, we only use Theorem~\ref{Tmult} from above.

\subsection{Proof of Theorem~\ref{Tmain} for finite fields}

The meat of the proof is in computing $SD(\mathbb{F})$ for $\mathbb{F}$ a
finite field. We first dispense with the exceptional case $\mathbb{F} =
\mathbb{F}_5$. In this case, Theorem~\ref{Tmult} gives $f(w) = w$ for $w
= 0, \pm 1$. As the SD-map $f$ is injective, it must fix or interchange
$2,3$. The former map is the (only) automorphism (and hence an SD-map)
${\rm id}_{\mathbb{F}_5}$, while the other map is precisely $f(w) \equiv
w^3$, which is an SD-map by Example~\ref{Ex5}.

Here is some \textbf{notation} for the remainder of this section. Fix an
odd prime $p>2$, and given a prime power $q = p^\ell$, let $\mathbb{F}_q$
denote the field with $q$ elements.

We now proceed. Fix a finite field $\mathbb{F}_q$ with $q$ odd, and a
generator $\theta$ of the cyclic group $\mathbb{F}_q^\times$. As $f :
\mathbb{F}_q \to \mathbb{F}_q$ is multiplicative and fixes $0,1$, it is
completely determined by $f(\theta)$; moreover, injectivity yields
$f(\theta) = \theta^k$ for some $k \geq 1$ coprime to
$|\mathbb{F}_q^\times| = q-1$. Thus $0 < k < q-1$, and $f(w) \equiv w^k$
on $\mathbb{F}_q$ by multiplicativity (and since $f(0) = 0$).

Now if $x=0$ then~\eqref{Esumdiff} is trivially true, so suppose $x \neq
0$ and replace $(x,y)$ by $(1, w = yx^{-1})$ (since $f$ is
multiplicative), to get that $f(w) = w^k$ satisfies:
\begin{equation}\label{Eratio}
\frac{(1+w)^k}{(1-w)^k} = \frac{1+w^k}{1-w^k}, \qquad \forall w \neq 1.
\end{equation}

Clearing the denominators, we obtain:
\[
p_k(w) := w^k ((1+w)^k + (1-w)^k) - ((1+w)^k - (1-w)^k) = 0.
\]

This polynomial clearly vanishes at $w=1$ as well, hence on all of
$\mathbb{F}_q$. Moreover, using the binomial theorem (and since $k$ is
odd, being coprime to $q-1$), it simplifies to
\[
2 w^k \left[ 1 + \binom{k}{2} w^2 + \cdots + \binom{k}{k-1} w^{k-1}
\right] - 2 \left[ \binom{k}{1} w^1 + \cdots + \binom{k}{k-2} w^{k-2} +
\binom{k}{k} w^k \right] = 0.
\]

The first and last terms cancel, and $p$ is odd so we can divide
throughout by $2$. This polynomial has degree $2k-1$ and $q$ roots in
$\mathbb{F}_q$. We now consider three cases for $k$.
\begin{enumerate}
\item First suppose $0 < k < (q+1)/2$, so $2k-1 < q$. Then $p_k(w) \equiv
0$ in $\mathbb{F}_q[w]$, so by symmetry and since $k$ is odd, the
binomial coefficients $\binom{k}{j} = 0 \mod p$ for $0 < j < k$. An
application of Lucas's theorem~\cite{Lucas} yields that $k$ must be a
power of $p$. That is, $k \in \{ 1, p, p^2, \dots, q/p \}$, since $q/p <
(q+1)/2$.

Conversely, if $k$ belongs to this set of $p$-powers, then indeed $w
\mapsto w^k$ is a field automorphism of $\mathbb{F}_q$, being a power of
the Frobenius. Thus, we are done if $k < (q+1)/2$.

\item The second case is that $(q+1)/2 < k < q-1$. We claim this is not
possible. Indeed, from~\eqref{Eratio} one can derive
\[
\frac{(1-w)^{-k}}{(1+w)^{-k}} = \frac{w^{-k} + 1}{w^{-k} - 1}, \qquad
\forall w \neq 0, \pm 1.
\]
But as $w^{q-1} \equiv 1$ on $\mathbb{F}_q^\times$, this translates into
\[
\frac{(1-w)^{q-1-k}}{(1+w)^{q-1-k}} = \frac{w^{q-1-k} + 1}{w^{q-1-k} -
1}, \qquad \forall w \neq 0, \pm 1.
\]

Writing $l := q-1-k$ (which is odd since $k$ is odd), this yields:
\[
\frac{(1-w)^l}{(1+w)^l} = \frac{w^l + 1}{w^l - 1}, \qquad \forall w \neq
0, \pm 1.
\]
Cross-multiplying, we get
\[
q_k(w) := w^l ((1+w)^l - (1-w)^l) + (1+w)^l + (1-w)^l = 0, \qquad \forall
w \neq 0,
\pm 1.
\]

Since $l$ is odd, $q_k$ is a polynomial of degree $2l$ (with leading term
$2 w^{2l} \neq 0$), which has $q-3$ distinct roots. But this contradicts
the fact that
\[
k > (q+1)/2 \quad \implies \quad 2l = 2q-2-2k < q-3.
\]

\item Thus, the only remaining case is $k = (q+1)/2$. (Notice, this
covers the exceptional case of $q=p=5$ and $k=3$.) Now work again with
the generator $\theta$ of $\mathbb{F}_q^\times$:
\[
f(\theta)^2 = f(\theta^2) = \theta^{2k} = \theta^{q+1} = \theta^q \cdot
\theta = \theta^2,
\]
since $w^q \equiv w$ on $\mathbb{F}_q$. But then $f(\theta) = \pm
\theta$. If $f(\theta) = \theta$ then $f \equiv {\rm id}_{\mathbb{F}_q}$
by multiplicativity. Otherwise $f(\theta) = -\theta$, so that $f(w) = \pm
w$ for each $w \in \mathbb{F}_q$.
Then set $x=1, y=\theta$ in~\eqref{Esumdiff} to get:
\[
\pm \frac{1+\theta}{1-\theta} = f \left( \frac{1+\theta}{1-\theta}
\right) = \frac{1+f(\theta)}{1-f(\theta)} = \frac{1-\theta}{1+\theta}.
\]

Solving both possible equations (via cross-multiplying) gives:
\[
4 \theta = 0 \qquad \text{or} \qquad 2(1+\theta^2) = 0.
\]
The former is not possible, so we get $1+\theta^2 = 0$, whence $\theta^2
= -1$ and so the generator $\theta$ is a primitive fourth root of unity
in $\mathbb{F}_q$. But then we recover precisely our exceptional case of
$q=p=5$! And here, the map $\theta \mapsto -\theta$ (for $\theta$ equal
to either $2$ or $3 \mod 5$) is indeed the map $w \mapsto w^3$, which is
an SD-map. \qed
\end{enumerate}

\subsection{Completing the proof for algebraic extensions}

To conclude the proof of Theorem~\ref{Tmain} (for algebraic field
extensions), we will require the following lemma.

\begin{lemma}\label{Lsubfield}
Suppose $p>2$ is odd and $q = p^\ell$ is a prime power (with $\ell \geq
1$). Let $f : \mathbb{F}_q \to \widetilde{\mathbb{F}}$ be an SD-map, with
${\rm char}(\widetilde{\mathbb{F}}) = p$. Then $f(\mathbb{F}_q)$ forms
the unique subfield of order $q$ inside $\widetilde{\mathbb{F}}$.
\end{lemma}

\begin{proof}
Let $\theta \in \mathbb{F}_q^\times$ be a generator. By
Theorem~\ref{Tmult}, $f$ is multiplicative and injective, and fixes
$0,1$, so $f(1) = 1, f(\theta), f(\theta^2), \dots, f(\theta^{q-2})$ are
distinct in $\widetilde{\mathbb{F}}$, and are $(q-1)$st roots of unity.
Hence they and $f(0) = 0$ are the only $q$ roots of the polynomial
$w^q-w$ in $\widetilde{\mathbb{F}}$. But it is well-known that these
roots inside any characteristic-$p$ field $\widetilde{\mathbb{F}}$, form
its unique $q$-element subfield.
\end{proof}

Now we complete the proof of Theorem~\ref{Tmain}. Suppose $\mathbb{F}_p
\subset \mathbb{F} \subseteq \overline{\mathbb{F}_p}$ is an infinite
algebraic field extension of $\mathbb{F}_p$, for some $p>2$ odd. Via
Theorem~\ref{Tmult}, it suffices to show that $f$ is
(a)~additive and
(b)~surjective.

We first place all primes (including $p=5$) on an ``equal'' footing:
choose any element $\alpha \in \mathbb{F} \setminus \mathbb{F}_p$; then
$\alpha$ is algebraic, so $\mathbb{F}_p(\alpha)$ is a finite extension,
say $\mathbb{F}_q$ with $q>p$. Moreover, $\mathbb{F}$ is an (infinite)
extension of $\mathbb{F}_q$.

Next, Lemma~\ref{Lsubfield} says that the restriction $f : \mathbb{F}_q
\to \mathbb{F}$ (with $q>5$) bijectively maps $\mathbb{F}_q$ onto itself.
By the preceding subsection, $f|_{\mathbb{F}_q}$ is an automorphism, and
hence additive on $\mathbb{F}_q$.

Now we show that $f$ is in fact additive on all of $\mathbb{F}$. Take any
$a,b$ not both in $\mathbb{F}_q$; as $a,b$ are algebraic, the extension
$\mathbb{E} := \mathbb{F}_q(a,b)$ is a finite field. Again by
Lemma~\ref{Lsubfield}, $f$ sends $\mathbb{E}$ bijectively onto itself.
But then by the preceding subsection, $f$ is additive on $\mathbb{E}$,
and so $f(a+b) = f(a)+f(b)$. This shows (a)~additivity.

The proof of (b)~surjectivity is similar. Let $a \in \mathbb{F} \setminus
\mathbb{F}_q$; then $\mathbb{F}_q(a)$ is a finite field. By
Lemma~\ref{Lsubfield}, it is mapped bijectively onto itself by $f$, and
so $f$ is onto. \qed

\subsection{A parallel result}

The proof of Theorem~\ref{Tmain} (for finite fields) in fact shows a
parallel result: \textit{
If $\widetilde{\mathbb{F}}$ is a field of characteristic $p>2$, and $f :
\mathbb{F}_p \to \widetilde{\mathbb{F}}$ is an SD-map, then either $p
\neq 5$ and $f \equiv {\rm id}$, or $p=5$ and $f(w) \equiv w$ or $w^3$.}

Using Lemma~\ref{Lsubfield}, this statement extends to finite fields:

\begin{proposition}\label{Pprev}
Suppose $\widetilde{\mathbb{F}}$ is a field of characteristic $p>2$, and
$q = p^\ell$ for some $\ell \geq 1$. Let $f : \mathbb{F}_q \to
\widetilde{\mathbb{F}}$ be an SD-map. Then $f$ is a field automorphism of
$\mathbb{F}_q$, except for $q=p=5$, in which case $f$ can alternately be
only the map $w \mapsto w^3$.
\end{proposition}

In turn, this is a special case of a parallel result to
Theorem~\ref{Tmain}, which we now show for completeness:

\begin{theorem}
Suppose $\widetilde{\mathbb{F}}$ is a field of characteristic $p>2$, and
$\mathbb{F}$ is algebraic over $\mathbb{F}_p$. Let $f : \mathbb{F} \to
\widetilde{\mathbb{F}}$ be an SD-map. Then $f$ is a field homomorphism of
$\mathbb{F}$ onto its image, except when $\mathbb{F} = \mathbb{F}_5$, in
which case $f$ can alternately be only the map $w \mapsto w^3$.
\end{theorem}

\begin{proof}
If $\mathbb{F}$ is finite, then we are (reduced to
Proposition~\ref{Pprev}, and) done by Lemma~\ref{Lsubfield}: $f$ may be
taken to be a self-SD-map of $\mathbb{F}$, and then we are done by
Theorem~\ref{Tmain}.

Else $\mathbb{F}$ is infinite. By Theorem~\ref{Tmult}, it remains to show
$f$ is additive. This is shown verbatim to the proof of
Theorem~\ref{Tmain}: first place all primes on an ``equal'' footing, so
$\mathbb{F} \supset \mathbb{F}_q$ for some $q>p$. Now show $f$ is
additive on $\mathbb{F}_q$ (via Lemma~\ref{Lsubfield} and
Theorem~\ref{Tmain}), and then on $\mathbb{F}_q(a,b)$ for all $a,b \in
\mathbb{F}$.
\end{proof}

\section{Additional characterizations of $\mathbb{F}_5$}

The final case~(3) in the proof of Theorem~\ref{Tmain} reveals the
uniqueness of the finite field $\mathbb{F}_5$. We distill the analysis
into multiple characterizations of $\mathbb{F}_5$ -- unique among
not just finite fields but also algebraic ones.

\begin{corollary}
Let $\mathbb{F}$ be a field of odd characteristic that is algebraic over
its prime subfield. Then the following are equivalent.
\begin{enumerate}
\item $\mathbb{F} = \mathbb{F}_5$.
\item $\Aut(\mathbb{F}) \subsetneq SD(\mathbb{F})$.
\item (For finite $\mathbb{F}$:) The map which fixes all squares in
$\mathbb{F}$ and sends all non-squares to their additive inverses is an
SD-map.
\item $\sqrt{-1}$ generates $\mathbb{F}^\times$.
\end{enumerate}
\end{corollary}

As above, this includes all finite fields of odd order. Also note that if
in assertion~(3) we allow $\mathbb{F}$ to be infinite, then the
characterization fails. For instance, suppose $\mathbb{F} =
\overline{\mathbb{F}_p}$, or the ``constructible numbers'' over
$\mathbb{F}_p$ (which form an infinite field closed under taking square
roots). Then the map in~(3) is the identity automorphism, which is an
SD-map.

\begin{proof}
That $(1) \Leftrightarrow (2)$ is Theorem~\ref{Tmain}, and that $(1)
\Rightarrow (3),(4)$ is easily verified. Conversely, if~(4) holds then
$\mathbb{F}^\times$ has order $4 = o(\sqrt{-1})$, so $|\mathbb{F}| = 5$
as desired.

Finally, say~(3) holds, with $\mathbb{F} = \mathbb{F}_q$ finite, and fix
a generator $\theta$ of $\mathbb{F}^\times$. Then $\theta$ is not a
square in $\mathbb{F}^\times$.\footnote{Indeed, any square root is of the
form $\theta^n \in \mathbb{F}_q^\times$. But then $\theta = \theta^{2n}$, so $2n \equiv 1 \mod (q-1)$. This is false as $q$ is odd.}
Now the non-squares in $\mathbb{F}$ are all odd powers of $\theta$, and
the squares are the rest. Thus, the proof of case~(3) now reveals (via
applying~\eqref{Esumdiff} to $\theta, 1$) that $\theta$ has order~4, and
so $\mathbb{F} = \mathbb{F}_5$.
\end{proof}

We now characterize $\mathbb{F}_5$ from an alternate viewpoint. Namely,
for $\mathbb{F}_5$ we had $w^3 \in SD(\mathbb{F}_5) \setminus
\Aut(\mathbb{F}_5)$; and this held for no other field algebraic over
$\mathbb{F}_p$ for $p$ odd. We now show more strongly that this holds for
no other fields of characteristic $p$ odd (or $p=0$):

\begin{proof}[Proof of Theorem~\ref{Tiff5}]
That $(2) \implies (1)$ is by Example~\ref{Ex5}. Conversely, if ${\rm
char}(\mathbb{F}) = 3$ then $w^3$ is the Frobenius automorphism, so~(1)
would fail. Now Theorem~\ref{Tmult} says that $f$ fixes $0,1$ and is
multiplicative, so setting $y=1$ in~\eqref{Esumdiff} yields:
\[
\frac{(x+1)^3}{(x-1)^3} = \frac{x^3+1}{x^3-1}, \qquad \forall x \in
\mathbb{F} \setminus \{ 1 \}.
\]
Cross-multiplying and simplifying gives: $6x^5 - 6x = 0$, so $x^5-x=0$
for all $x \neq 1$ (since ${\rm char}(\mathbb{F}) \neq 2,3$). But this
equation also holds for $x=1$, so a degree-5 equation has $|\mathbb{F}|$
roots. Hence $|\mathbb{F}| \leq 5$. As ${\rm char}(\mathbb{F}) \neq 2,3$,
we get $\mathbb{F} = \mathbb{F}_5$.
\end{proof}

\begin{remark}
As the proof reveals, another equivalent condition is that ${\rm
char}(\mathbb{F}) \neq 2,3$ and $w^3 \in SD(\mathbb{F})$.
\end{remark}

It is natural to wonder about the missing case in Theorem~\ref{Tiff5}:
what happens if ${\rm char}(\mathbb{F})=2$. In this case, not all fields
have $w^3$ as an SD-map:

\begin{proposition}\label{Pcube}
Let $\mathbb{F}$ be a field of characteristic $2$. Then $f(w) \equiv w^3
\in SD(\mathbb{F})$ if and only if $\mathbb{F}$ does not contain a
quadratic extension of $\mathbb{F}_2$.
\end{proposition}

For instance, if $\mathbb{F} = \mathbb{F}_{2^\ell}$ is finite, then $w^3
\in SD(\mathbb{F})$ if and only if $\ell$ is odd.

\begin{proof}
All quadratic extensions of $\mathbb{F}_2$ are isomorphic to
$\mathbb{F}_4$. Thus, suppose $\mathbb{F} \supseteq \mathbb{F}_4$, and
let $\theta \in \mathbb{F}_4 \setminus \mathbb{F}_2$. Then $\theta$ is a
primitive generator of $\mathbb{F}_4^\times$, so $f(\theta) = 1 = f(1)$
and hence $f$ is not injective. By Proposition~\ref{Pchar2}, $w^3$ is
therefore not an SD-map (hence, not a field homomorphism either).

Conversely, say $w \mapsto w^3$ is not an SD-map. By
Proposition~\ref{Pchar2}, this is equivalent to $w^3$ not being
injective. Since $0$ is the only cube root of $0$, there exist $x \neq y$
in $\mathbb{F}^\times$ such that $x^3 = y^3$. Setting $w := xy^{-1}$ as
usual, we obtain $w \in \mathbb{F}^\times \setminus \{ 1 \}$ such that
$w^3 = 1$. But then $1, w, w^2, 0$ are pairwise distinct, and are roots
of $x^4-x$ over $\mathbb{F}$. Thus they form a subfield of $\mathbb{F}$
of order $4$.
\end{proof}

\section{Which powers are SD-maps -- of which fields?}

We conclude with a classification question that arises naturally from
Theorem~\ref{Tiff5} and Proposition~\ref{Pcube}. Namely, we saw above
that $\mathbb{F} = \mathbb{F}_5$ is the only finite/algebraic field for
which $SD(\mathbb{F}) \supsetneq \Aut(\mathbb{F})$; moreover, the
difference is the lone map $w \mapsto w^3$. We then studied the situation
``dually'', from the viewpoint of the cube map: it is an SD-map of
\textit{any} field $\mathbb{F}$ if and only if exactly one of the
following occurs:
\begin{itemize}
\item $\mathbb{F} = \mathbb{F}_5$;
\item ${\rm char}(\mathbb{F}) = 3$; or
\item ${\rm char}(\mathbb{F}) = 2$ and $\mathbb{F} \not\supset \mathbb{F}_4$.
\end{itemize}

It is natural to ask what is the analogous classification (of all fields)
for the power map $w \mapsto w^m$. Here $m > 0$, because one applies this
map to $w=0$ as well.

\begin{theorem}\label{Twhichpowers}
Let $m$ be a positive integer, and $\mathbb{F}$ a field. The map $w
\mapsto w^m$ is in $SD(\mathbb{F})$ if and only if exactly one of the
following occurs:
\begin{enumerate}
\item $m=1$ and $\mathbb{F}$ is arbitrary.

\item $m>1$, $\mathbb{F} = \mathbb{F}_q$ is finite with characteristic $p
> 2$ and size $q = p^\ell$, and $m \equiv p^k \mod (q-1)$ for some $0
\leq k \leq \ell-1$.

\item $m>1$, and $\mathbb{F} = \mathbb{F}_{2^\ell}$ is finite, with $m$
coprime to $2^\ell-1$ (e.g., $m = 1, 2, \dots, 2^{\ell-1}$).

\item $m>1$, $\mathbb{F}$ is infinite with characteristic $2$; now write
$m = 2^a m'$ where $a \geq 0$ and $m' \geq 3$ is odd. Then $\mathbb{F}$
does not have a nontrivial $m'$th root of unity.

\item $\mathbb{F}$ is infinite with characteristic $p \geq 2$, and $m =
p^k$ for some $k \geq 1$.

\item $m=3$ and $\mathbb{F} = \mathbb{F}_5$.
\end{enumerate}
\end{theorem}

Thus, the final case of $(w^3, \mathbb{F}_5)$ is again exceptional in
this result: all other cases are part of a family of examples/cases,
while $(w^3, \mathbb{F}_5)$ is not.
Note also that for fields of characteristic zero, only $m=1$ is an SD-map
(and an automorphism).

\begin{proof}
Note that there are no overlaps between the six cases. Now we show the
backward implication.
Clearly~(1) or~(5) imply that the field map $w^m \in SD(\mathbb{F})$, as
does~(6) by Example~\ref{Ex5}. If instead~(2) holds, then $w^q \equiv w$
on $\mathbb{F}_q$, so one can reduce the exponent by $q-1$ at a time, and
get to $w^{p^k}$, which is a power of the Frobenius, hence in
$\Aut(\mathbb{F}_q) \subseteq SD(\mathbb{F}_q)$.

Next, say~(3) $\mathbb{F} = \mathbb{F}_{2^\ell}$ and denote the cyclic
generator of $\mathbb{F}_{2^\ell}^\times$ by $\theta$. Then $\theta$ has
order $2^\ell-1$, hence so does $\theta^m$ by hypothesis. As the power
map is multiplicative, $w^m$ fixes $0,1$ and permutes the other elements,
so it is an SD-map by Proposition~\ref{Pchar2}.

Finally, suppose the conditions in~(4) hold but $w^m \not\in
SD(\mathbb{F})$. By Proposition~\ref{Pchar2}, there exist $x \neq y$ in
$\mathbb{F}$ such that $x^m = y^m$. If $x=0$ then $y=0$, so we must have
$x,y \neq 0$. Replacing $(x,y)$ by $(w = xy^{-1}, 1)$, it follows that $w
\neq 1$ but $w^m = 1$. Now write $m = 2^a m'$; applying the Frobenius,
\[
w^{2^a} - 1^{2^a} = (w-1)^{2^a} \neq 0.
\]
Setting $w_o := w^{2^a} \neq 1$, we have $w_o^{m'} = w^m = 1$, a
contradiction.

This shows that each of the six conditions implies $w^3$ is an SD-map. We
now show the converse. Clearly $m=1$ works, so we suppose henceforth that
$m>1$. First say $\mathbb{F} = \mathbb{F}_q$ is finite, and let $\theta$
denote the cyclic generator of $\mathbb{F}_q^\times$. As SD-maps are
injective, $\theta^m$ generates $\mathbb{F}_q^\times$, so $(m,q-1) = 1$.
This shows~(3) if $\mathbb{F}_q$ has even order. Else if ${\rm
char}(\mathbb{F}_q)$ is odd, then as $w^q \equiv w$ on $\mathbb{F}_q$,
one can reduce powers $q-1$ at a time, to assume without loss of
generality that $m \in (0,q-1)$. Now Theorem~\ref{Tmain} addresses
cases~(2) and~(6).

This leaves us to deduce~(4) or~(5) when $\mathbb{F}$ is infinite. First
note since $m>1$ that if ${\rm char}(\mathbb{F}) \in \{ 0 \} \cup
(2^m,\infty)$, then $w^m$ is not an SD-map by Theorem~\ref{Tchar}, since
it does not fix the prime subfield -- because $(1+1)^m \neq 1^m + 1^m$.
Thus, ${\rm char}(\mathbb{F}) \leq 2^m$.

There are now two cases. First say ${\rm char}(\mathbb{F}) = 2$, and
suppose $m$ is not a power of $2$ (so we need to prove~(4)). Suppose for
contradiction that $\mathbb{F}$ contains a nontrivial $m'$th root of
unity, say $w_o$. Then $w_o \neq 1$ but $w_o^{m'} = 1$, so $w_o^m = 1$
too. Thus $w \mapsto w^m$ is not injective, hence not an SD-map by
Proposition~\ref{Pchar2}.
The other case is where $\mathbb{F}$ is infinite and ${\rm
char}(\mathbb{F})$ is odd. Now $w^m$ satisfies~\eqref{Eratio}, and
repeating the subsequent analysis in case~(1) shows (via Lucas's theorem)
that $m$ is a power of $p$.
\end{proof}

Note that Theorem~\ref{Twhichpowers}(4) for $m=m'=3$ is formulated in a
different way than Proposition~\ref{Pcube}. Yet, they are both the same,
since a field of characteristic $2$ contains $\mathbb{F}_4$ if and only
if it contains (a cyclic generator of $\mathbb{F}_4^\times$, which is the
same as) a nontrivial cube root of unity. Moreover, they are both
equivalent to $w^3$ not being one-to-one.

These equivalences hold more generally -- for all exponents $m$ and
characteristics $p \geq 2$:

\begin{proposition}
Let $m, p \geq 2$ be integers with $p$ prime, and $\mathbb{F}$ be a field
of characteristic $p$. Write $m = p^a m'$ where $p \nmid m'$. Then the
following are equivalent:
\begin{enumerate}
\item The map $w \mapsto w^m$ is not one-to-one in $\mathbb{F}$.

\item $\mathbb{F}$ contains a nontrivial $m$th root of unity.

\item $\mathbb{F}$ contains a nontrivial $m'$th root of unity.
\end{enumerate}

If $m'$ is prime, these are further equivalent to:
\begin{enumerate}
\setcounter{enumi}{3}
\item $\mathbb{F}$ contains a copy of the finite field
$\mathbb{F}_{p^o}$, where $o = o_p(m')$ denotes the order of $p$ in the
group of units $(\mathbb{Z} / m' \mathbb{Z})^\times$.
\end{enumerate}
\end{proposition}

This provides (for $p=2$) several reformulations of
Theorem~\ref{Twhichpowers}(4) -- the last of them is how
Theorem~\ref{Twhichpowers} subsumes Proposition~\ref{Pcube}.

\begin{proof}
If~(1) holds then there exist $x \neq y$ in $\mathbb{F}$ such that $x^m =
y^m$. As $x=y$ if $y=0$, both $x,y$ must be nonzero. Now one can replace
$(x,y)$ by $(w = xy^{-1},1)$ to deduce~(2). Next if~(2) holds with say $w
\neq 1$ and $w^m = 1$, then applying the Frobenius,
\[
w^{p^a} - 1^{p^a} = (w-1)^{p^a} \neq 0.
\]
Thus $w' := w^{p^a} \neq 1$, and yet $(w')^{m'} = w^m = 1$ -- which
shows~(3). Now if~(3) holds for some $w'$, then $(w')^{m'} = 1$, so
$(w')^m = 1$, showing~(1).

Now say $m'$ is prime. Then~(3) equivalently says:
$(3')$~$\mathbb{F}$ contains a primitive $m'$th root of unity, say
$\zeta_{m'}$.\footnote{Recall using the Frobenius that there are no
primitive $m$th roots of unity over $\mathbb{F}_p$ (i.e., in
$\overline{\mathbb{F}_p}$) if $p|m$.}
It thus remains to show $(3') \Leftrightarrow (4)$.
More generally, we explain why for $m'$ is coprime to $p$, a
characteristic-$p$ field $\mathbb{F}$ containing a primitive $m'$th root
of unity $\zeta_{m'}$ is equivalent to $\mathbb{F} \supseteq
\mathbb{F}_{p^o}$, with $o = o_p(m')$ as in the result. One way is clear:
if $o$ is as defined, then $|\mathbb{F}_{p^o}^\times| = p^o-1$ is
divisible by $m'$. Denoting their quotient by $n'$, and letting $\theta$
be a cyclic generator of $\mathbb{F}_{p^o}^\times$, we see that
$\zeta_{m'} := \theta^{n'}$ works.
Conversely, $\mathbb{F} \ni \zeta_{m'}$ is equivalent to $\mathbb{F}$
containing the splitting field of $\zeta_{m'}$ -- i.e.\ of the $m'$th
cyclotomic polynomial over $\mathbb{F}_p$. Let this field be
$\mathbb{F}_p(\zeta_{m'}) \equiv \mathbb{F}_{p^\ell}$ for some $\ell >
0$, say. Then the order of $\zeta_{m'}$ divides $p^\ell-1$, i.e.\ $p^\ell
\equiv 1 \mod m$. Thus $o|\ell$, and so $\mathbb{F}_{p^o} \subseteq
\mathbb{F}_{p^\ell}$. As $\zeta_{m'} \in \mathbb{F}_{p^o}$ from above,
the minimality of the splitting field gives $\ell = o$.
\end{proof}

\subsection{Concluding remarks: Why bijections?}

Recall that we started the paper by defining a novel group over every
field $\mathbb{F}$, which contained $\Aut(\mathbb{F})$. We now end the
paper by defining a novel monoid over every field, which contains
$SD(\mathbb{F})$. In order for solutions $f$ of~\eqref{Esumdiff} to form
(the SD-) group, surjectivity has to be imposed on $f$. However, we could
alternately have considered only the monoid of solutions:
\begin{equation}
SD_+(\mathbb{F}) := \{ f : \mathbb{F} \to \mathbb{F} \ | \ f \textit{
satisfies } \eqref{Esumdiff} \}.
\end{equation}

Now the group $SD(\mathbb{F})$ is a sub-monoid of $SD_+(\mathbb{F})$. If
$\mathbb{F}$ is a finite field, then clearly the two are equal, since
every injection is in fact a bijection. Thus, one can ask if the two are
equal for infinite fields.

This turns out to be false, and in every characteristic (and so a
question for further study can be to understand $SD_+(\mathbb{F})$). We
provide a few examples.
\begin{enumerate}
\item This example and the next work over any field -- and in particular,
in any characteristic. Let $\mathbb{F}'$ be any field, and let
$\mathbb{F} := \mathbb{F}'(x)$. Given $k \geq 1$, define the map $f_k : x
\mapsto x^k$, so $f_k((p(x)/q(x)) := p(x^k)/q(x^k)$ for $p, 0 \neq q \in
\mathbb{F}'[x]$. Then $f_2, f_3, \dots$ are all in $SD_+(\mathbb{F})
\setminus SD(\mathbb{F})$.

\item Let $\mathbb{F}'$ be any field, and consider its transcendental
extension and self-map:
\[
\mathbb{F} := \mathbb{F}'(x_0, x_1, x_2, \dots); \qquad f_+ : \mathbb{F}
\to \mathbb{F} \text{ sends } x_n \mapsto x_{n+1}\ \forall n \geq 0.
\]
This is a field monomorphism that is not an automorphism for any base
field $\mathbb{F}'$.

\item The final example (was stated in~\cite{CKS}, and) is in fact of a
subfield of $\mathbb{R}$! It shows that even in ``familiar'' situations
in characteristic zero, one can have field monomorphisms that are not
onto. Indeed, consider the subfield $\mathbb{F} = \mathbb{Q}(\pi)$ --
since $\pi$ is transcendental over $\mathbb{Q}$ -- and define for $k \in
\mathbb{Z}$ the map
\begin{equation}\label{Efk}
f_k : \frac{p(\pi)}{q(\pi)} \mapsto \frac{p(\pi^k)}{q(\pi^k)}, \qquad
p,q \in \mathbb{Q}[x].
\end{equation}
Then $f_2, f_3, \dots$ are all in $SD_+(\mathbb{F}) \setminus
SD(\mathbb{F})$.
\end{enumerate}

Thus, not every solution $f$ to~\eqref{Esumdiff} is an automorphism: it
need not be surjective (by the above), nor need it be additive
(Example~\ref{Ex5}).

We end by mentioning that the SD-group is a novel construction, and much
remains to be explored. Two natural follow-ups to Theorem~\ref{Tmain}
are:
(1)~Can one compute $SD(\mathbb{F}_p(t))$, and is it also equal to
$\Aut(\mathbb{F}_p(t))$, at least for $p \neq 5$? What if $p=5$?
(2)~What happens for finite algebraic extensions (of the prime field) in
characteristic zero? E.g., is $SD(\mathbb{F}) = \Aut(\mathbb{F})$ for a
number field $\mathbb{F}$? This was indeed verified for quadratic number
fields in~\cite{CKS} (as well as the closures $\overline{\mathbb{Q}}$ and
$\mathbb{R} \cap \overline{\mathbb{Q}}$). Larger extensions remain
unexplored.

\appendix
\section{SD-maps of the $p$-adics}

Here, we go beyond the results in~\cite{CKS} and classical literature,
and determine another SD-group. Recall from the final row of
Table~\ref{Ttable1} that the only continuous SD-map of $\mathbb{Q}_p$ is
the identity. However, as we know, the only SD-map of the (only) other
completion of $\mathbb{Q}$ -- namely of $\mathbb{R}$ -- is the identity
map, and this does not assume continuity. Thus, it is natural to try to
compute $SD(\mathbb{Q}_p)$. This is our final result, and it strengthens
the classical fact that the only field automorphism of the $p$-adic
numbers is the identity map.

\begin{theorem}\label{Tpadic}
$SD(\mathbb{Q}_p) = \{ {\rm id} \}$ for all primes $p \geq 2$.
\end{theorem}

The rest of this section proves the theorem. We begin by setting some
notation. Fix a prime integer $p \geq 2$, and identify the $p$-adics
$\mathbb{Q}_p$ as the completion of $\mathbb{Q}$ under the metric
\[
d(x,y) := p^{-\nu_p(x-y)} =: |x-y|_p,
\]
where given a rational $m/n$ with $p^a,p^b$ the largest prime powers
dividing $m,n$ respectively, we define the valuation $\nu_p(m/n) :=
p^{a-b}$. (Also define $\nu_p(0) := +\infty$; thus $|0|_p := 0$.)

The completion $\mathbb{Q}_p$ is a topological field that can be
bijectively identified with formal Laurent series in $p$:
\begin{equation}\label{Esqcup}
\mathbb{Q}_p = \{ 0 \} \sqcup \left\{ \sum_{n=k}^\infty a_n p^n : k \in
\mathbb{Z} , \ a_k \neq 0 , \ a_n \in \{ 0, \dots, p-1 \}\ \forall n
\right\}.
\end{equation}

Now the $p$-adic valuation and norm / absolute value extend from
$\mathbb{Q}$ to $\mathbb{Q}_p$ via:
\[
\nu_p \left( \sum_{n=k}^\infty a_n p^n \right) := p^{-k}, \quad
d(x,y) := p^{-\nu_p(x-y)} =: |x-y|_p.
\]

Moreover, recall the ring of $p$-adic integers, denoted by
$\mathbb{Z}_p$, is the subset of formal power series (so $k \geq 0$
above), together with $0$. The \textit{$p$-adic units} in $\mathbb{Z}_p$
are (identified with) the subset
\[
\mathbb{Z}_p^\times := \left\{ \sum_{n=0}^\infty a_n p^n : a_0 \in \{ 1,
\dots, p-1 \} \right\} = \nu_p^{-1}(0).
\]

We now prove Theorem~\ref{Tpadic}. This uses two well-known lemmas -- the
first is classical: 

\begin{lemma}[Hensel's (lifting) lemma, e.g.~\cite{Eisenbud}]\label{Lhensel}
Let $f(x) \in \mathbb{Z}_p[x]$. Suppose there exists $x_0 \in
\mathbb{Z}_p$ such that
\[
f(x_0) \equiv 0 \mod{p} \quad \text{and} \quad f'(x_0) \not\equiv 0
\mod{p}.
\]
Then there exists a unique $\widetilde{a} \in \mathbb{Z}_p $ such that
\[
f(\widetilde{a}) = 0 \quad \text{and} \quad \widetilde{a} \equiv x_0
\mod{p}.
\]
\end{lemma}

In turn, Hensel's lemma is used to show the following characterization of
$p$-adic units:

\begin{lemma}\label{Lunit}
Suppose $0 \neq u \in \mathbb{Q}_p$. Then $u \in \mathbb{Z}_p^\times$ if
and only if $u$ admits an $n$th root in $\mathbb{Q}_p$ for infinitely
many $n>0$.
\end{lemma}

\begin{proof}
We include a short proof for completeness.
Suppose $u$ admits an $n$th root for infinitely many $n_1 < n_2 <
\cdots$, say $x_k^{n_k} = u$ for all $k \geq 1$. Thus $x_k \neq 0$;
taking valuations yields $n_k \nu_p(x_k) = \nu_p(u) \in \mathbb{Z}$, so
$n_k | \nu_p(u)$ for all $k \geq 1$. This yields $\nu_p(u) = 0$, i.e.\ $u
\in \mathbb{Z}_p^\times$.

Conversely, define $n_k := 1 + k p (p-1)$ for all $k \geq 1$, and also
write $u \equiv a_0 \mod{p}$ (so $0 < a_0 < p$, and $u = a_0 + a_1 p +
\cdots$). As $(n_k, p-1) = 1$, there exists a $n_k$th root of $a_0$
modulo $p$, say $x_k \in \mathbb{Z} \setminus p \mathbb{Z}$. (In fact,
here $x_k = a_0 \ \forall k$ by Fermat's Little Theorem by choice of
$n_k$.) Letting $f(x) := x^{n_k} - u$, and since $p \nmid n_k$, we have:
\[
f(x_k) = x_k^{n_k} - u \equiv a_0 - u \equiv 0 \mod p, \qquad
f'(x_k) = n_k x_k^{n_k-1} \not\equiv 0 \mod p,
\]
Thus by Hensel's lemma~\ref{Lhensel}, there exists
$\widetilde{a}_k \in \mathbb{Z}_p$ such that $0 = f(\widetilde{a}_k) =
(\widetilde{a}_k)^{n_k} - u$.
\end{proof}

Finally, we show:

\begin{proof}[Proof of Theorem~\ref{Tpadic}]
The proof strategy is:
(i)~$f$ is an SD-map on $\mathbb{Q}_p$
$\implies$ (ii)~$f$ sends $p$-adic units to themselves
$\implies$ (iii)~$f$ sends $0$ to $0$ and $\nu_p^{-1}(k)$ into itself for every integer $k$
$\implies$ (iv)~$f$ is continuous at $0$
$\implies$ (v)~$f$ is continuous at $1$
$\implies$ (vi)~$f$ is continuous everywhere
$\implies$ (vii)~$f$ is the identity map.\medskip

\noindent \textit{(i) $\implies$ (ii):}
If $u \in \mathbb{Z}_p^\times$ then by Lemma~\ref{Lunit} there exist
infinitely many integers $n_k > 0$ and roots $x_k \in \mathbb{Q}_p$ such
that $x_k^{n_k} = u$. Since $f$ is multiplicative (by
Theorem~\ref{Tmult}), $f(x_k)^{n_k} = f(u)$ for all $k$. Again by
Lemma~\ref{Lunit}, $f(u) \in \mathbb{Z}_p^\times$.\medskip

\noindent \textit{(ii) $\implies$ (iii)} and \textit{(vi) $\implies$ (vii):}
Since ${\rm char}(\mathbb{Q}_p) = 0$, Theorem~\ref{Tchar} implies that
$f$ fixes $\mathbb{Q}$. This shows by density that~(vi) implies~(vii).
Moreover, $f(p) = p$, so by~(ii) and multiplicativity, we get $f(p^k u) =
p^k f(u)$ for all $k \in \mathbb{Z}$ and units $u \in
\mathbb{Z}_p^\times$. Since~\eqref{Esqcup} can now be written as
\[
\mathbb{Q}_p = \{ 0 \} \sqcup \bigsqcup_{k \in \mathbb{Z}} \nu_p^{-1}(-k)
= \{ 0 \} \sqcup \bigsqcup_{k \in \mathbb{Z}} p^k \mathbb{Z}_p^\times,
\]
it follows from above that $f$ respects this partition of
$\mathbb{Q}_p$.\medskip

\noindent \textit{(iii) $\implies$ (iv):}
Note that $p^k \mathbb{Z}_p = \{ 0 \} \sqcup \bigsqcup_{m \geq k} p^m
\mathbb{Z}_p^\times$. Now since $f$ respects the partition above, it
follows that $f^{-1}(p^k \mathbb{Z}_p) \subseteq p^k \mathbb{Z}_p$ for
all $k \in \mathbb{Z}$. Combining with~(iii) gives equality:
$f^{-1}(p^k \mathbb{Z}_p) = p^k \mathbb{Z}_p\ \forall k$.
Since the sets $p^k\mathbb{Z}_p$ form a fundamental basis of open sets
around the origin in the $p$-adic topology, this shows that $f$ is
continuous at $0$.\medskip

\noindent \textit{(iv) $\implies$ (v):}
Take any sequence $t_n \to 1$ in $\mathbb{Q}_p$; we may assume $t_n \neq
-1$ for all $n$. Set $s_n := \frac{t_n-1}{t_n+1}$; then $t_n \to 1
\Leftrightarrow s_n \to 0$. Now compute, using~(iv) and that $f$ fixes
$0,1$:
\[
t_n = \frac{1 + s_n}{1 - s_n} \quad \implies \quad \lim_{n \to \infty}
f(t_n) = \lim_{n \to \infty} \frac{1 + f(s_n)}{1 - f(s_n)} =
\frac{1+0}{1-0} = f(1).
\]

\noindent \textit{(v) $\implies$ (vi):}
Finally, let $0 \neq a \in \mathbb{Q}_p$ and let $r_n \to a$ in
$\mathbb{Q}_p$. Then $a^{-1} r_n \to 1$, so by multiplicativity,
\[
\lim_{n \to \infty} f(r_n) = 
\lim_{n \to \infty} f(a) f(a^{-1}r_n) = f(a) \cdot 1.
\]
Thus $f$ is continuous on $\mathbb{Q}_p$, and hence fixes all of
$\mathbb{Q}_p$ as mentioned above.
\end{proof}

We end with a natural question for future consideration: What is
$SD(\mathbb{F})$ for a subfield $\mathbb{Q} \subsetneq \mathbb{F}
\subsetneq \mathbb{Q}_p$? Note that $\mathbb{Q}, \mathbb{Q}_p$ have
trivial SD-groups (and hence trivial automorphism groups), but the same
need not hold for $\mathbb{F}$ without extra assumptions (and in
particular, the above proof does not carry over for arbitrary
$\mathbb{F}$). For example, by a theorem of Cassels~\cite{Cassels}, every
number field (i.e.\ finite extension of $\mathbb{Q}$) embeds inside
$\mathbb{Q}_p$ for infinitely many primes $p$. So its SD-group contains
its Galois group over $\mathbb{Q}$, which can indeed be nontrivial.

The situation here is similar to the case of $\mathbb{R}$, where we had
to assume continuity to show that every SD-map on $\mathbb{F}$ is
trivial, where $\mathbb{Q} \subsetneq \mathbb{F} \subsetneq \mathbb{R}$.

\end{document}